\documentclass[final,nomarks]{dmtcs-episciences}

\usepackage[utf8]{inputenc}
\usepackage{subfigure}

\usepackage{url}
\usepackage{tikz}
\usepackage{forest}
\usepackage{cancel}
\usepackage[numbers,sort&compress]{natbib}
\usepackage{amsmath,amsthm}

\newcommand{\RRT}{\mbox{RRT}}
\newcommand{\RT}{\mbox{RT}}

\renewcommand{\tt}[1]{\textnormal{\texttt{#1}}}

\newtheorem{theorem}{Theorem}
\newtheorem{corollary}[theorem]{Corollary}
\newtheorem{lemma}[theorem]{Lemma}

\newtheorem{observation}[theorem]{Observation}

\theoremstyle{definition}

\title{The repetition threshold for binary rich words}

\author{James D. Currie\thanks{The work of James D. Currie is supported by the Natural Sciences and Engineering Research Council of Canada (NSERC), [funding reference number 2017-03901].}
\and Lucas Mol
\and Narad Rampersad\thanks{The work of Narad Rampersad is supported by the Natural Sciences and Engineering Research Council of Canada (NSERC), [funding reference number 2019-04111].}
}
\affiliation{Department of Mathematics and Statistics, 
University of Winnipeg \\
515 Portage Avenue,
Winnipeg, Manitoba R3B 2E9 (Canada)}


\received{2019-9-27}
\revised{2020-1-10}
\accepted{2020-1-13}

\begin{document}
\publicationdetails{22}{2020}{1}{6}{5791}
\maketitle
\begin{abstract}
  A word of length $n$ is rich if it contains $n$ nonempty palindromic factors.  An infinite word is rich if all of its finite factors are rich.  Baranwal and Shallit produced an infinite binary rich word with critical exponent $2+\sqrt{2}/2$ ($\approx 2.707$) and conjectured that this was the least possible critical exponent for infinite binary rich words (i.e., that the repetition threshold for binary rich words is $2+\sqrt{2}/2$).  In this article, we give a structure theorem for infinite binary rich words that avoid $14/5$-powers (i.e., repetitions with exponent at least $2.8$).  As a consequence, we deduce that the repetition threshold for binary rich words is $2+\sqrt{2}/2$, as conjectured by Baranwal and Shallit.  This resolves an open problem of Vesti for the binary alphabet; the problem remains open for larger alphabets.
\end{abstract}

\section{Introduction}

A \emph{palindrome} is a word that is equal to its reversal, i.e., it reads the same forwards and backwards.  It is well-known that a word of length $n$ contains at most $n$ distinct nonempty palindromes~\cite{DJP01}.  Words of length $n$ that contain $n$ distinct nonempty palindromes are called \emph{palindrome-rich}, or simply \emph{rich}.  An infinite word is \emph{rich} if all of its factors are rich.  Rich words were introduced in~\cite{BHNR04} (where they were called \emph{full} words), were first studied systematically in~\cite{GJWZ09}, and have since been studied by many authors~\cite{DGZ08,PS13,RR09,V14,V17}.

Let $u$ be a finite nonempty word, and let $u=u_1\dots u_{n},$ where the $u_i$ are letters.  A positive integer $p$ is a \emph{period} of $u$ if $u_{i}=u_{i+p}$ for all $1\leq i\leq n-p$.  Let $e=|u|/p$ and let $z$ be the prefix of $u$ of length $p$.  Then we say that $e$ is an \emph{exponent} of $p$, and write $u=z^e$.  We say that $u$ is \emph{primitive} if the only integer exponent of $u$ is $1$.

For a real number $\alpha\geq 1$, a finite or infinite word $w$ is called \emph{$\alpha$-free} if it contains no nonempty factor of exponent greater than or equal to $\alpha$.  Otherwise, we say that $w$ \emph{contains} an $\alpha$-power.  The \emph{critical exponent} of $w$ is the supremum of the set of all rational numbers $\alpha$ such that $w$ contains an $\alpha$-power.
The \emph{repetition threshold} for a language $L$ is the infimum of the set of all real numbers $\alpha>1$ such that there is an infinite $\alpha$-free word in $L$.
In other words, the repetition threshold for $L$ is the smallest possible critical exponent among all infinite words in $L$.

The repetition threshold for the language of all words on a fixed alphabet of size $k$, denoted $\RT(k)$, was introduced by Dejean~\cite{Dej72}, who conjectured that
\[
\RT(k)=\begin{cases}
2, &\text{if $k=2$;}\\
7/4, &\text{if $k=3$;}\\
7/5, &\text{if $k=4$;}\\
k/(k-1), &\text{if $k \geq 5$.}
\end{cases}
\]
This conjecture was eventually proven through the work of many authors~\cite{Carpi07,CR09,CR11,Dej72,MC07,MO92,Pansiot84,Rao11}.  
Rampersad et al.~\cite{RSV19} recently proposed the problem of determining the repetition threshold for the language of \emph{balanced} words over a fixed alphabet of size $k$.  Both Rampersad et al.~\cite{RSV19} and Baranwal and Shallit~\cite{BS19balanced} have made some progress on this problem.
% denoted $\BRT(k)$, and made some progress on the problem.  The fact that $\BRT(2)=2+\varphi$, where $\varphi$ is the golden ratio, follows from well-known results on Sturmian words (see~\cite{RSV19}).  Rampersad et al.~\cite{RSV19} demonstrated that $\BRT(3)=2+\sqrt{2}/2$ and $1.8088\leq \BRT(4)\leq 1+\varphi/2\approx 1.8090$, and conjectured that $\BRT(k)=(k-2)/(k-3)$ for all $k\geq 5$.  Baranwal and Shallit~\cite{BS19balanced} confirmed this conjecture for $k=5$.

We are concerned with repetitions in rich words.  Vesti~\cite{V17} proposed the problem of determining the repetition threshold for the language of rich words over $k$ letters, denoted $\RRT(k)$. Vesti noted that $2\leq \RRT(k)\leq 2+1/(\varphi_k-1)$ for all $k\geq 2$, where $\varphi_k$ is the generalized golden ratio.  The lower bound follows from the fact that every infinite rich word contains a square~\cite{PS13}.  The upper bound follows from the fact that the $k$-bonacci word is rich and has critical exponent $2+1/(\varphi_k-1)$~\cite{GJ09}. Baranwal and Shallit~\cite{BS19} demonstrated that there is an infinite binary rich word with critical exponent $2+\sqrt{2}/2$, and conjectured that this is the smallest possible critical exponent among all infinite binary rich words, i.e., that $\RRT(2)=2+\sqrt{2}/2$.  In this article, we prove a structure theorem for infinite $14/5$-free binary rich words.  We use this theorem to confirm Baranwal and Shallit's conjecture.

We use the following notation throughout the paper.  Let $\Sigma_k=\{\tt{0},\tt{1},\dots,\tt{k-1}\}$. Define $f:\Sigma_3^*\rightarrow\Sigma_2^*$ and $g,h:\Sigma_3^*\rightarrow\Sigma_3^*$ by
\begin{align*}
    f(\tt{0})&=\tt{0}\\
    f(\tt{1})&=\tt{01}\\
    f(\tt{2})&=\tt{011}\\[5pt]
    g(\tt{0})&=\tt{011}\\
    g(\tt{1})&=\tt{0121}\\
    g(\tt{2})&=\tt{012121}\\[5pt]
    h(\tt{0})&=\tt{01}\\
    h(\tt{1})&=\tt{02}\\
    h(\tt{2})&=\tt{022}
\end{align*}
Note that $f(h^\omega(\tt{0}))$ is the infinite binary rich word with
critical exponent $2+\sqrt{2}/2$ constructed by Baranwal and
Shallit~\cite{BS19}.  Also, note that $f$, $g$, and $h$ are
injective.  Furthermore, these three morphisms all belong to the
well-studied family of \emph{class~$P$} morphisms \cite{HKS95}, which are
connected to the study of rich words \cite{BPS11}.

We prove the following structure theorem for infinite $14/5$-free binary rich words.\footnote{Note that $g=\tilde{g}\circ h$, where $\tilde{g}:\Sigma_{3}^*\rightarrow \Sigma_2^*$ is defined by $\tilde{g}(\tt{0})=\tt{01}$, $\tilde{g}(\tt{1})=\tt{1}$, and $\tilde{g}(\tt{2})=\tt{21}.$  Thus, in the statement of Theorem~\ref{structure}, one could replace $g$ with $\tilde{g}$.  For convenience, we have elected to work with the morphism $g$ throughout.}

\begin{theorem}\label{structure}
Let $w\in\Sigma_2^\omega$ be a $14/5$-free rich word.  For every $n\geq 1$, a suffix of $w$ has the form $f(h^n(w_n))$ or $f(g(h^n(w_n)))$ for some word $w_n\in\Sigma_3^\omega$.
\end{theorem}

\noindent
We then demonstrate that, like $f(h^\omega(\tt{0}))$, the word $f(g(h^\omega(\tt{0})))$ has  critical exponent $2+\sqrt{2}/2$.  This gives the following.

\begin{theorem}\label{threshold}
The repetition threshold for binary rich words is $2+\sqrt{2}/2$.
\end{theorem}

Our structure theorem is somewhat reminiscent of the well-known structure theorem for overlap-free binary words due to Restivo and Salemi~\cite{RS83,RS85}, and its extension to $7/3$-free binary words by Karhum\"aki and Shallit~\cite{KS04}.  However, we deal only with infinite words.

\section{A structure theorem}

In this section, we prove Theorem~\ref{structure}.  Throughout, we say
that a word $w\in\Sigma_2^\omega$ is \emph{good} if it is both rich
and $14/5$-free.  In particular, a good word is cube-free.

We begin by proving several properties of the morphisms $f$, $g$, and $h$.  For every $\phi\in\{f,g,h\}$, one verifies by computer using a straightforward backtracking algorithm that the longest word $u\in\{\tt{1},\tt{2}\}^*$ such that $\phi(u)$ is cube-free has length $6$.  This gives the following.

\begin{observation}\label{start}
Let $\phi\in\{f,g,h\}$ and $u\in\Sigma_3^\omega$.  If $\phi(u)$ is cube-free, then $u$ contains a \tt{0}.
\end{observation}

We now show that the morphisms $f$, $g$, and $h$ preserve non-richness
of $\omega$-words.  We require two short lemmas.  The first can be
derived from \cite[Lemma~5.2]{BPS11}, but we give a proof here for completeness.

\begin{lemma}\label{pal suf}
Let $\phi\in\{f,g,h\}$ and let $u, v\in\Sigma_3^*$. 
Suppose $\phi(u)\tt{0}$ is a palindromic suffix of $\phi(v)\tt{0}$. Then $u$ is a palindromic suffix of $v$.
\end{lemma}
\begin{proof}
%The single \tt{0} beginning each letter's image under $\phi$ forces $\phi$ to be synchronizing.
Since $\phi(u)\tt{0}$ is a suffix of $\phi(v)\tt{0}$ and $\phi$ is injective, we have that $u$ is a suffix of $v$. For any $u\in\Sigma_3^*$, we have $\tt{0}(\phi(u))^R=\phi(u^R)\tt{0}$. Since $\phi(u)\tt{0}$ is a palindrome,
$(\phi(u)\tt{0})^R=\tt{0}(\phi(u))^R=\phi(u^R)\tt{0}$. Since $\phi$ is injective, we have $u=u^R$. 
Thus $u$ is a palindromic suffix of $v$.
\end{proof}

In order to prove the next lemma, we use the fact that a word $w$ is rich if and only if every nonempty prefix $p$ of $w$ has a nonempty palindromic suffix that appears only once in $p$~\cite{GJWZ09}.

\begin{lemma}\label{non-rich} Let $\phi\in\{f,g,h\}$. Suppose that $w\in\Sigma_3^*$ is non-rich. Then $\phi(w)\tt{0}$ is non-rich.
\end{lemma}
\begin{proof}
Let $w'$ be a prefix of $w$ such that every palindromic suffix of $w'$ occurs at least twice in $w'$.  We claim that $\phi(w')\tt{0}$ is a prefix of $\phi(w)\tt{0}$ such that every palindromic suffix of $\phi(w')\tt{0}$ occurs at least twice in $\phi(w')\tt{0}$.  Any palindromic suffix of $\phi(w')\tt{0}$ has the form $\phi(u)\tt{0}$ for some $u$.  Then by Lemma~\ref{pal suf}, we know that $u$ is a palindromic suffix of $w'$.  However, by hypothesis, $w'$ contains two occurrences of $u$.  Consequently, $\phi(w')\tt{0}$ contains two occurrences of the palindrome $\phi(u)\tt{0}$.  We conclude that $\phi(w)\tt{0}$ is non-rich, as required.
\end{proof}

The fact that the morphisms $f$, $g$, and $h$ preserve non-richness of $\omega$-words now follows as an easy corollary.

\begin{corollary}\label{non-rich-cor}
Let $\phi\in\{f,g,h\}$ and $u\in\Sigma_3^\omega$.  If $\phi(u)$ is rich, then $u$ is rich.
\end{corollary}

By straightforward induction arguments using Observation~\ref{start} and Corollary~\ref{non-rich-cor}, we obtain the following.

\begin{lemma}\label{basic}
Let $\phi$ be a morphism of the form $f\circ h^n$ or $f\circ g\circ h^n$ for some $n\geq 0$.  If $\phi(u)$ is good for some $u\in\Sigma_3^\omega$, then the word $u$ is cube-free, rich, and contains a \tt{0}.
\end{lemma}

\noindent
We use Lemma~\ref{basic} frequently throughout this section, sometimes without reference.

If $w$ is good, then $w$ avoids the cube $\tt{111}$, so the following observation is immediate.

\begin{observation}\label{obs}
If $w\in\Sigma_2^\omega$ is good, then a suffix of $w$ has the form $f(u)$ for some word $u\in\Sigma_3^\omega$.
\end{observation}

So we may now restrict our attention to good words of the form $f(u)$,
where $u\in\Sigma_3^\omega$.  By Lemma~\ref{basic}, if
$u\in\Sigma_3^\omega$ is a word such that $f(u)$ is good, then every
factor of $u$ is rich, i.e., no non-rich word is a factor of $u$.
There are a variety of other short factors that cannot appear in such
a word $u$.  One checks by backtracking that for each word $v$ in
Table~\ref{table}, there is a longest right-extension
$vs\in\Sigma_3^*$ of $v$ such that $f(vs)$ is not
good. Table~\ref{table} indicates in each case the length of such a
longest extension $vs$. (The notation $*$ indicates that $f(v)$
already fails to be good.)  Hence, none of the factors in
Table~\ref{table} can appear in $u\in\Sigma_3^\omega$ if $f(u)$ is
good.  We use this fact frequently throughout this section.  We also
remark that the choice of the constant $14/5$ in the definition of
``good'' becomes relevant at this backtracking step.  If we
replace $14/5$ with $3$ in the definition of ``good'', then for
certain $v$ the backtracking search runs for a very long time without
finding a longest right-extension $vs$ such that $f(vs)$ is not good.

\begin{table}
\begin{center}
 $\begin{array}{|l|l|l|}\hline
\mbox{Table row}&v&|vs|\\\hline
1&\tt{00}&2\\
2&\tt{0121012}&49\\
3&\tt{021}&22\\
4&\tt{0221}&19\\
5&\tt{11010}&24\\
6&\tt{11011}&29\\
7&\tt{1102}&30\\
8&\tt{112}&*\\
9&\tt{120}&22\\
10&\tt{122}&17\\
11&\tt{21010}&6\\
12&\tt{2101210}&48\\
% \cancel{13}&\tt{2102}&*\\
13&\tt{211}&3\\\hline
\end{array}$
\end{center}
\caption{Forbidden factors in every $\omega$-word $u$ such that $f(u)$ is good.}
\label{table}
\end{table}

We will prove that if $f(u)$ is good for some $\omega$-word $u$, then $u$ either has a suffix of the form $g(W)$, or a suffix of the form $h(W)$.  It turns out that if $u$ contains the factor \tt{0110}, then we are forced into the former structure.  Otherwise, if $h$ does not contain \tt{0110}, then we are forced into the latter structure.  We handle the case that $u$ contains the factor \tt{0110} first.  In fact, we show that in this case, a suffix of $u$ must have the form $h(g(U))$.

\begin{lemma}~\label{0110} Suppose $f(u)$ is good, where $u\in\Sigma_3^\omega$, and $u$ contains the factor \tt{0110}. Then \begin{enumerate}
\item The word $u$ has a suffix of the form $g(W)$ for some word $W\in\Sigma_3^\omega$.
\item A suffix of $W$ has the form $h(U)$ for some word $U\in\Sigma_3^\omega$.
\end{enumerate}
\end{lemma}
\begin{proof}
(1) Replacing $u$ by a suffix if necessary, write $u=u_1u_2u_3u_4\cdots$, where $u_1=\tt{011}$ and each $u_i$ starts with \tt{0} and contains no other \tt{0}. To show that $u=g(W)$ for some $W\in\Sigma_3^\omega$, it will suffice to show that every $u_i$ is one of \tt{011}, \tt{0121} or \tt{012121}.  The proof is by induction on $i$.  The base case is immediate since $u_1=\tt{011}$.

Now suppose for some $i\geq 1$ that $u_i\in\{\tt{011},\tt{0121},\tt{012121}\}$. Consider the tree in Figure~\ref{tree1}, which shows all candidates for $u_{i+1}\tt{0}$.  We explain why the word ending at every unboxed leaf of the tree cannot be a prefix of $u_{i+1}\tt{0}$, from which we conclude that $u_{i+1}\in\{\tt{011},\tt{0121},\tt{012121}\}$.  We use the following facts: 
\begin{itemize}
    \item By Lemma~\ref{basic}, the word $u$ is cube-free and rich.
    \item No word in Table~\ref{table} is a factor of $u$.
    \item The word $u_i$ must have suffix \tt{11} or \tt{21} by the induction hypothesis; so if neither $\tt{11}x$ nor $\tt{21}x$ appears in $u$, then $x$ cannot be a prefix of $u_{i+1}\tt{0}$.
\end{itemize}
We discuss each unboxed leaf of the tree in lexicographic order.
\begin{itemize}
    \item \tt{00}: The word \tt{00} is in Table~\ref{table}.
    \item \tt{010}: The words \tt{11010} and \tt{21010} are in Table~\ref{table}.
    \item \tt{0111}: The word \tt{111} is a cube.
    \item \tt{0112}: The word \tt{112} is in Table~\ref{table}.
    \item \tt{0120}: The word \tt{0120} is not rich.
    \item \tt{01211}: The word \tt{211} is in Table~\ref{table}.
    \item \tt{012120}: The word \tt{012120} is not rich.
    \item \tt{0121211}: The word \tt{211} is in Table~\ref{table}.
    \item \tt{0121212}: The word \tt{121212} is a cube.
    \item \tt{012122}: The word \tt{122} is in Table~\ref{table}.
    \item \tt{0122}: The word \tt{122} is in Table~\ref{table}.
    \item \tt{02}: The word \tt{1102} is in Table~\ref{table}, and the word \tt{2102} is not rich.
\end{itemize}

\begin{figure}
\centering
\begin{forest}
[
\tt{0},for tree={grow=0,l=2cm}
    [
    \tt{2}
    ]
    [
    \tt{1}
        [
        \tt{2}
            [
            \tt{2}
            ]
            [
            \tt{1}
                [
                \tt{2}
                    [
                    \tt{2}
                    ]
                    [
                    \tt{1}
                        [
                        \tt{2}
                        ]
                        [
                        \tt{1}
                        ]
                        [
                        \tt{0},draw,fill=green
                        ]
                    ]
                    [
                    \tt{0}
                    ]
                ]
                [
                \tt{1}
                ]
                [
                \tt{0},draw,fill=green
                ]
            ]
            [
            \tt{0}
            ]
        ]
        [
        \tt{1}
            [
            \tt{2}
            ]
            [
            \tt{1}
            ]
            [
            \tt{0},draw,fill=green
            ]
        ]
        [
        \tt{0}
        ]
    ]
    [
    \tt{0}
    ]
]
\end{forest}
    \caption{The tree showing all possible prefixes of $u_{i+1}\tt{0}$.}
    \label{tree1}
\end{figure}
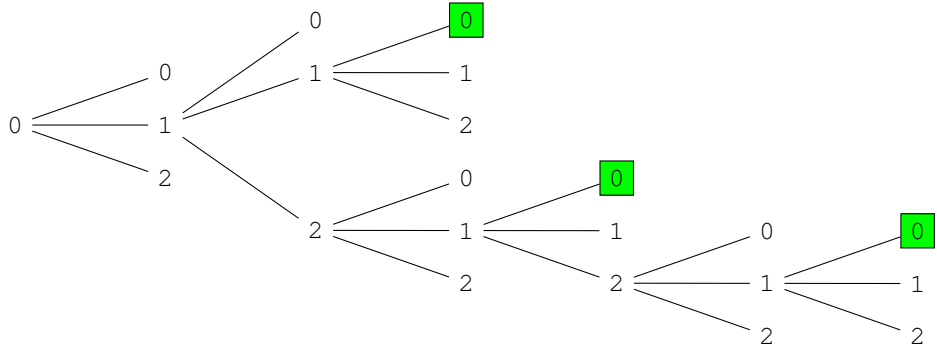

(2) To begin with, we show that \tt{00}, \tt{11}, \tt{12}, and \tt{21} are not factors of $W$. If $W$ contains \tt{00}, then $u$ contains $g(\tt{00})=\tt{011011}$, but this is impossible since \tt{11011} is in Table~\ref{table}.  If $W$ contains \tt{11} or \tt{12}, then $u$ contains $g(\tt{11})=\tt{01210121}$ or $g(\tt{12})=\tt{0121012121}$; but this is impossible since \tt{0121012} is in Table~\ref{table}. Finally, if $W$ contains \tt{21}, then $u$ contains $g(\tt{21})\tt{0}$=\tt{01212101210}, but this is impossible since \tt{2101210} is in Table~\ref{table}.

% Since \tt{11} and \tt{12} are not factors of $W$, we see that \tt{1} is always followed by \tt{0} in $W$. Therefore, if $W$ doesn't contain \tt{0} as a factor, then $W=\tt{2}^\omega$; this is impossible, since $W$ is cube-free. 
By Lemma~\ref{basic}, the word $W$ contains a \tt{0}. Replacing $W$ by a suffix if necessary, write $W=W_1W_2W_3W_4\cdots$, where each $W_i$ starts with \tt{0} and contains no other \tt{0}.  Let $i\geq 1$.  As above, we enumerate the possible prefixes of $W_i0$ in the tree of Figure~\ref{tree2}.  It is easy to verify that the word ending at every unboxed leaf of the tree ends in one of the factors \tt{00}, \tt{11}, \tt{12}, \tt{21}, or the cube \tt{222}, so we conclude that $W_i\in\{\tt{01},\tt{02},\tt{022}\}$ as desired.
\begin{figure}
\centering
\begin{forest}
[
\tt{0},for tree={grow=0,l=2cm}
    [
    \tt{2}
        [
        \tt{2}
            [
            \tt{2}
            ]
            [
            \tt{1}
            ]
            [
            \tt{0},draw,fill=green
            ]
        ]
        [
        \tt{1}
        ]
        [
        \tt{0},draw,fill=green
        ]
    ]
    [
    \tt{1}
        [
        \tt{2}
        ]
        [
        \tt{1}
        ]
        [
        \tt{0},draw,fill=green
        ]
    ]
    [
    \tt{0}
    ]
]
\end{forest}
    \caption{The tree showing all possible prefixes of $W_i\tt{0}$.}
    \label{tree2}
\end{figure}
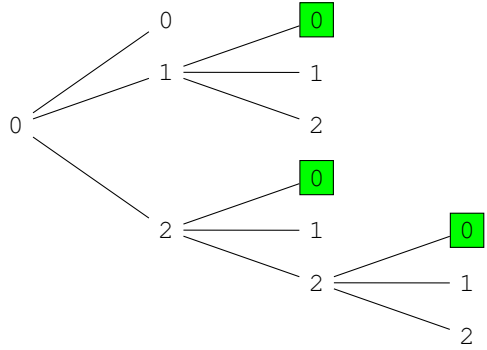
% Since \tt{00} doesn't appear in $W$, each $W_i$ has prefix either \tt{01} or \tt{02}. If $W_i$ has prefix \tt{01}, then $W_i=\tt{01}$, since $W$ has no factor \tt{11} or \tt{12}. Suppose $W_i$ has prefix \tt{02}, but $|W_i|>2$; then, since \tt{21} is not a factor of $W$, we see that $W_i$ has prefix \tt{022}. As neither of \tt{21} nor the cube \tt{222} is a factor of $W$, we conclude that $W_i=\tt{022}$.
\end{proof}

We now show that there are several factors that do not appear in relevant preimages of good words.  Define $F=\{\tt{1221}, \tt{00}, \tt{10101}, \tt{212},\tt{11}\}$.

\begin{lemma}\label{forbidden factors}
Let $u\in\Sigma_3^\omega$. Suppose that for some positive integer $n$, one of $f(g(h^n(u)))$ and $f(h^n(u))$ is good. Then a suffix of $u$ does not contain any of the factors in the set $F$.
\end{lemma}
\begin{proof}
By Lemma~\ref{basic}, the word $u$ must be cube-free and rich, and we may assume, by taking a suffix if necessary, that $u$ begins in \tt{0}.

\smallskip
\noindent {\tt{1221}:}
Since $h(\tt{1221})$ contains a cube, \tt{1221} cannot be a factor of $u$. 

\smallskip
\noindent{\tt{00}:}
For any letter $x\in\{0,1,2\}$, all of $f(h(\tt{00}x))$, $g(h(\tt{00}x))$, and $h^2(\tt{00}x)$ contain a cube.  Suppose towards a contradiction that \tt{00} is a factor of $u$.  If $n=1$, then $f(h(u))$ and $f(g(h(u)))$ contain factors of the form $f(h(\tt{00}x))$ and $f(g(h(\tt{00}x)))$, respectively, giving a cube. Otherwise, if $n\ge 2$, then $f(h^n(u))$ and $f(g(h^n(u)))$ contain factors of the form $f(h^{n-2}(h^2(\tt{00}x)))$ and $f(g(h^{n-2}(h^2(\tt{00}x))))$, respectively, giving a cube.  Since $u$ is cube-free, this is impossible, and we conclude that \tt{00} is a not a factor of $u$.

\smallskip
\noindent{\tt{10101}:}
All of $f(h(\tt{10101}x))$, $g(h(\tt{10101}x))$, and $h^2(\tt{10101}x)$ contain cubes.  By an argument similar to the one used for \tt{00}, we see that the factor \tt{10101} cannot be a factor of $u$.

\smallskip
\noindent {\tt{212}:}
First note that $f(\tt{0})$ is a prefix of $f(\tt{1})$, which is a prefix of $f(\tt{2})$. It follows that if $v\in\Sigma_3^*$, then $f(v\tt{0}v\tt{0}v\tt{1})$ and $f(v\tt{1}v\tt{1}v\tt{2})$ contain cubes. Next, note that $g(v\tt{0}v\tt{0}v\tt{2})=(V\tt{1}V\tt{1}V\tt{2})\tt{121}$, where $V=g(v)\tt{01}$.  Since $g(\tt{1})$ is a prefix of $g(\tt{2})$, we see that $g(v\tt{1}v\tt{1}v\tt{2})$ contains a cube. Similarly, note that $h(v\tt{0}v\tt{0}v\tt{2})=(V\tt{1}V\tt{1}V\tt{2})\tt{2}$, where $V=h(v)\tt{0}$.  Further, since $h(\tt{1})$ is a prefix of $h(\tt{2})$, we see that $h(v\tt{1}v\tt{1}v\tt{2})$ contains a cube. Finally, note that $h(\tt{212})=\tt{02202022}$ ends in a factor of the form $v\tt{0}v\tt{0}v\tt{2}$, where $v=\tt{2}$. 

Suppose that \tt{212} is a factor of $u$. It follows by induction that  $h^n(u)$ contains either a cube, a factor of the form $v\tt{0}v\tt{0}v\tt{2}$ (in the case $n=1$), or a factor of the form $v\tt{1}v\tt{1}v\tt{2}$. It follows that $g(h^n(u))$ contains a factor of the form $V\tt{1}V\tt{1}V\tt{2}$, or a cube, so that $f(h^n(u))$ and $f(g(h^n(u)))$ both contain cubes. This is impossible.

% \smallskip
% \noindent {\tt{0120}, \tt{2012}, \tt{01220}, \tt{0210}, \tt{02210}, \tt{210102}:} In each of these words, the longest palindromic suffix has length 1, and occurs earlier in the word.  Hence each of these words is not rich. By Lemma~\ref{basic}, none of them is a factor of $u$.

\smallskip
\noindent {\tt{11}:} Suppose that $\tt{11}$ is a factor of $u$.  The words \tt{111}, $h(\tt{112})$, and $h(\tt{211})\tt{0}$ all contain a cube, hence $\tt{11}$ is preceded and followed by \tt{0}. Thus, \tt{0110} is a factor of $u$. However, all of $f(h(\tt{0110}))$, $g(h(\tt{0110}))$, and $h^2(\tt{0110})$ contain a cube.  By an argument similar to the one used for \tt{00}, we conclude that $\tt{11}$ is not a factor of $u$.
\end{proof}

We now prove that any cube-free rich word $u\in\Sigma_3^\omega$ that avoids the finite list of factors from Lemma~\ref{forbidden factors} must have a suffix of the form $h(W)$.  Together, Lemma~\ref{forbidden factors} and the following lemma will form the inductive step of our structure theorem.

\begin{lemma}\label{0-blocks} Suppose that $u\in\Sigma_3^\omega$ is cube-free and rich. If $u$ does not contain any of the factors in the set $F$,
%=\{\tt{1221},\tt{00},\tt{10101},\tt{212},\tt{11}\}$,
then $u$ has a suffix of the form $h(W)$ for some word $W\in\Sigma_3^\omega$.
\end{lemma}
\begin{proof} 
Taking a suffix of $u$ if necessary, write $u=u_1u_2u_3u_4\cdots$, where each $u_i$ starts with \tt{0} and contains no other \tt{0}. It will suffice to show that every $u_i$ is one of \tt{01}, \tt{02} or \tt{022}.  For an arbitrary $i\geq 1$, as in the proof of Lemma~\ref{0110}, we consider the tree of possible prefixes of $u_i\tt{0}$, drawn in Figure~\ref{0-blocks-tree}.  We explain why the word ending at every unboxed leaf of the tree cannot be a factor of $u$.
\begin{itemize}
    \item \tt{00}: The word \tt{00} is in $F$.
    \item \tt{011}: The word \tt{11} is in $F$.
    \item \tt{0120}: The word \tt{0120} is not rich.
    \item \tt{01211}: The word \tt{11} is in $F$.
    \item \tt{01212}: The word \tt{212} is in $F$.
    \item \tt{01220}: The word \tt{01220} is not rich.
    \item \tt{01221}: The word \tt{1221} is in $F$.
    \item \tt{01222}: The word \tt{222} is a cube.
    \item \tt{0210}: The word \tt{0210} is not rich.
    \item \tt{0211}: The word \tt{11} is in $F$.
    \item \tt{0212}: The word \tt{212} is in $F$.
    \item \tt{02210}: The word \tt{02210} is not rich.
    \item \tt{02211}: The word \tt{11} is in $F$.
    \item \tt{02212}: The word \tt{212} is in $F$.
    \item \tt{0222}: The word \tt{222} is a cube.
\end{itemize}

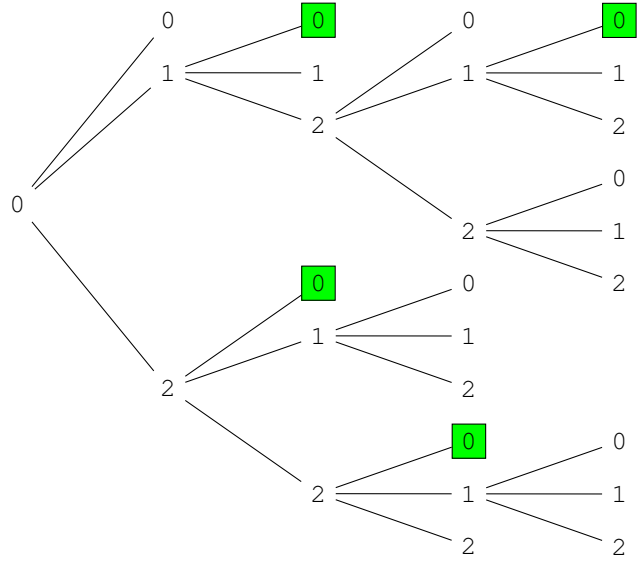
\begin{figure}
\centering
\begin{forest}
[
{\tt{0}}, for tree={grow=0,l=2cm}
    [
    \tt{2}
        [
        \tt{2}
            [
            \tt{2}
            ]
            [
            \tt{1}
            [
                \tt{2}
                ]
                [
                \tt{1}
                ]
                [
                \tt{0}
                ]
            ]
            [
            \tt{0},draw,fill=green
            ]
        ]
        [
        \tt{1}
            [
            \tt{2}
            ]
            [
            \tt{1}
            ]
            [
            \tt{0}
            ]
        ]
        [
        \tt{0},draw,fill=green
        ]
    ]
    [
    \tt{1}
        [
        \tt{2}
            [
            \tt{2}
                [
                \tt{2}
                ]
                [
                \tt{1}
                ]
                [
                \tt{0}
                ]
            ]
            [
            \tt{1}
                [
                \tt{2}
                ]
                [
                \tt{1}
                ]
                [
                \tt{0},draw,fill=green
                ]
            ]
            [
            \tt{0}
            ]
        ]
        [
        \tt{1}
        ]
        [
        \tt{0},draw,fill=green
        ]
    ]
    [
    \tt{0}
    ]
]
\end{forest}
    \caption{The tree showing all possible prefixes of $u_i\tt{0}$.}
    \label{0-blocks-tree}
\end{figure}

Thus, we conclude from Figure~\ref{0-blocks-tree} that $u_i\in\{\tt{01},\tt{0121}, \tt{02}, \tt{022}\}$ for all $i\geq 1$.  Suppose towards a contradiction that for some $i\geq 1$, we have $u_i=\tt{0121}$. Because $u$ does not have the non-rich word \tt{2102} as a factor, we see that $u_{i+1}\ne \tt{02}, \tt{022}$. Suppose that $u_{i+1}=\tt{01}$. Then $u_{i+2}\in\{\tt{01},\tt{0121}, \tt{02}, \tt{022}\}$,
forcing $u$ to contain one of \tt{10101}, or \tt{210102}. However, this is impossible since $\tt{10101}$ is in $F$, and $\tt{210102}$ is not rich. We conclude that $u_{i+1}=\tt{0121}$. By the same argument, $u_{i+2}=\tt{0121}$, and $u$ contains the cube $(\tt{0121})^3$. This is impossible. It follows that we cannot have $u_i=\tt{0121}$, so that $u_i\in\{\tt{01},\tt{02},\tt{022}\}$, as desired.
\end{proof}

Finally, we still need to handle the case that $f(u)$ is good, but $u\in\Sigma_3^*$ does not contain the factor \tt{0110}.

\begin{lemma}\label{No0110}
Suppose $f(u)$ is good for some word $u\in\Sigma_3^\omega$ that does not contain the factor \tt{0110}. Then $u$ has a suffix of the form $h(W)$. 
\end{lemma}
\begin{proof}
By Lemma~\ref{basic}, we know that $u$ is cube-free and rich, and by
taking a suffix if necessary, we may assume that $u$ begins in \tt{0}.
By Lemma~\ref{0-blocks}, it suffices to show that $u$ does not contain any of the words
in $F$.

\smallskip

\noindent \tt{1221,00}: The words \tt{122} and \tt{00} are in Table~\ref{table}.

\smallskip

\noindent \tt{10101}: Since $f(\tt{0})$ is a prefix of $f(\tt{1})$ and $f(\tt{2})$, the word $f(\tt{10101}x)$ begins with a cube.  Since $f(u)$ is cube-free, we conclude that $u$ cannot contain the factor \tt{10101}.

\smallskip

\noindent \tt{212}: Backtracking by computer as we did to create Table~\ref{table}, with the additional restriction that \tt{0110} is not allowed, one finds that the longest right extension of \tt{212} has length $21$.  Hence \tt{212} is not a factor of $u$.

\smallskip

\noindent \tt{11}: The word \tt{11} cannot be preceded or followed by \tt{1} in $u$, since $u$ is cube-free.  Further, the word \tt{11} cannot be preceded or followed by \tt{2} in $u$, since \tt{112} and \tt{211} are in Table~\ref{table}. However, then if \tt{11} is a factor of $u$, so is \tt{0110}, contrary to assumption.
% Suppose the word $\tt{212}$ appears in $u$. Extending right, referring to rows 9 and 10 of the table, $u$ must contain 2121. Because of row 3, extending on the left can only give one of 22121 and 12121. The first of these has left extensions 022121, 122121 and 222121, which cannot be in $u$ because of rows 4 and 10 of the table, and containing a cube, respectively. It follows that $u$ contains a factor 12121. Since 121212 is a cube, and 11 is not a factor of $u$, this extends on the right to 121210. Since 2102 is not rich, and 00 is excluded by row 1, this extends on the right to 1212101. Since 11 cannot appear in $u$, this extends on the right as 12121010, or 12121012; the first is impossible by row 11. We conclude that $u$ contains the factor 12121012, which extends, by rows 9 and 10 to 121210121. Since 11 is not a factor of $u$, this can't be right-extended by a 1. By row 12, this can't be right-extended by a 0; thus, it extends to 1212101212. Now however, 212 is a suffix, and the argument iterates, and $u$ becomes ultimately periodic, which repeated period 121210. This is impossible.
\end{proof}

We are now ready to prove our structure theorem.

\begin{proof}[Proof of Theorem~\ref{structure}]
The proof is by induction on $n$.  We first establish the base case $n=1$.  By Observation~\ref{obs}, a suffix of $w$ has the form $f(w_0)$ for some word $w_0\in\Sigma_3^\omega$.  If $w_0$ contains the factor \tt{0110}, then by Lemma~\ref{0110}, there is a suffix of $w_0$ that has the form $g(h(w_1)).$  Otherwise, if $w_0$ does not contain the factor \tt{0110}, then by Lemma~\ref{No0110}, there is a suffix of $w_0$ that has the form $h(w_1)$.  Therefore, a suffix of $w$ has the form $f(h(w_1))$ or $f(g(h(w_1)))$ for some $w_1\in\Sigma_3^*$, establishing the base case.

Suppose now that for some $n\geq 1$, a suffix of $w$ has the form $f(h^n(w_n))$ or $f(g(h^n(w_n)))$ for some $w_n\in\Sigma_3^\omega$. By Lemma~\ref{forbidden factors}, there is a suffix of $w_n$ that does not contain any of the factors in $F=\{\tt{1221}, \tt{00}, \tt{10101}, \tt{212}, \tt{11}\}$.  By Lemma~\ref{basic}, we know that $w_n$ is cube-free and rich.  Therefore, by Lemma~\ref{0-blocks}, a suffix of $w_n$ has the form $h(w_{n+1})$ for some $w_{n+1}\in\Sigma_3^\omega$.  We conclude that a suffix of $w$ has the form $f(h^{n+1}(w_{n+1}))$ or $f(g(h^{n+1}(w_{n+1})))$.
\end{proof}

\section{The repetition threshold}

Baranwal and Shallit~\cite{BS19} showed that the word $f(h^\omega(\tt{0}))$ is rich and has critical exponent $2+\sqrt{2}/2$.  They showed both properties using the Walnut theorem prover.  We show that the word $f(g(h^\omega(\tt{0})))$ has the same properties using a different method, which relies heavily on a connection to Sturmian words; it turns out that both $f(h^\omega(\tt{0}))$ and $f(g(h^\omega(\tt{0})))$ are \emph{complementary symmetric Rote words}\footnote{This very useful observation was communicated to us by Edita Pelantov\'a.}.

A word $w\in \Sigma_2^\omega$ is a \emph{complementary symmetric Rote word} if its factorial language is closed under complementation and it has factor complexity $\mathcal{C}(n)=2n$ for all $n\geq 1$.  For any infinite binary word $w = (w_n)_{n \geq 0}$, let $\Delta(w) =((w_n+w_{n+1}) \bmod 2)_{n \geq 0}$, i.e., $\Delta(w)$ is the sequence of first differences of $w$ modulo $2$.  We use the fact that a word $w\in\Sigma_2^\omega$ is a complementary symmetric Rote word if and only if $\Delta(w)$ is a Sturmian word~\cite[Theorem~3]{Rot94}.

Let $u=f(g(h^\omega(\tt{0})))$.  We begin by showing that $\Delta(u)$ is a certain Sturmian word $v$, from which we conclude that $u$ is a complementary symmetric Rote word.  In particular, this implies that $u$ is rich~\cite{BMBLV11}.  We then relate the repetitions in $v$ to those in $u$, and use the theory of repetitions in Sturmian words to establish that the critical exponent of $u$ is $2+\sqrt{2}/2$.  We note that a similar calculation would provide an alternate proof of Baranwal and Shallit's result that the critical exponent of $f(h^\omega(\tt{0}))$ is $2+\sqrt{2}/2$.

Define $\lambda,\mu:\Sigma_3^*\rightarrow\Sigma_2^*$ by
\begin{align*}
    \lambda(\tt{0})&=\tt{0}\\
    \lambda(\tt{1})&=\tt{11}\\
    \lambda(\tt{2})&=\tt{101}\\[5pt]
    \mu(\tt{0})&=\tt{01111}\\
    \mu(\tt{1})&=\tt{01110111}\\
    \mu(\tt{2})&=\tt{0111011110111}
\end{align*}
We extend the map $\Delta$ to finite binary words in the obvious manner in order to prove the following straightforward lemma.

\begin{lemma}\label{DeltaMorphisms}
Let $w\in\Sigma_3^*$.  Then
\begin{enumerate}
    \item $\Delta(f(w)\tt{0})=\lambda(w)$, and
    \item $\Delta(f(g(w))\tt{0})=\mu(w)$.
\end{enumerate}
\end{lemma}

\begin{proof}
One checks that $\Delta(f(a)\tt{0})=\lambda(a)$ and $\Delta(f(g(a))\tt{0})=\mu(a)$ for all $a\in\Sigma_3$. 

For (1), we proceed by induction on the length $n$ of $w$.  When $n=0$, we have $\Delta(f(\varepsilon)\tt{0})=\varepsilon=\lambda(\varepsilon)$, so the statement holds.  Suppose for some $n\geq 0$ that the statement holds for all words $w$ of length $n$.  Let $x$ be a word of length $n+1$.  Then $x=ya$ for some $y\in\Sigma_3^n$ and $a\in\Sigma_3$.  Then $\Delta(f(x)\tt{0})=\Delta(f(y)\tt{0})\Delta(f(a)\tt{0})=\lambda(y)\lambda(a)=\lambda(x).$ 

The proof of (2) is similar.
\end{proof}

Define morphisms $\xi,\eta : \Sigma_2^* \to \Sigma_2^*$ by 
\begin{align*}
    \xi(\tt{0})&=\tt{011}\\
    \xi(\tt{1})&=\tt{01}\\[5pt]
    \eta(\tt{0})&=\tt{011}\\
    \eta(\tt{1})&=\tt{1}.
\end{align*}
Note that both $\xi$ and $\eta$ are Sturmian morphisms (see \cite[Section~2.3]{Lot02}).
By checking the images of all letters in $\Sigma_3$, one verifies that $\lambda\circ h=\xi\circ \lambda$ and $\mu=\eta\circ\xi\circ \lambda$.
% i.e., for all $w\in\Sigma_3^*$, we have $\lambda(h(w))=\xi(\lambda(w))$ and $\mu(w)=\eta(\xi(\lambda(w)))$.

\begin{lemma}\label{SturmianMorphisms}
\begin{enumerate}
    \item $\Delta(f(h^\omega(\tt{0})))=\xi^\omega(\tt{0})$
    \item $\Delta(f(g(h^\omega(\tt{0}))))=\eta(\xi^\omega(\tt{0}))$
\end{enumerate}
\end{lemma}

\begin{proof}
For (1), we show that $\Delta(f(h^n(\tt{0}))\tt{0})=\xi^n(\tt{0})$ for every $n\geq 0$.  First of all, we have $\Delta(f(h^n(\tt{0}))\tt{0})=\lambda(h^n(\tt{0}))$ by Lemma~\ref{DeltaMorphisms}, so it suffices to show that $\lambda(h^n(\tt{0}))=\xi^n(\tt{0})$.  We proceed by induction on $n$.  The statement is easily verified when $n=0$.  Suppose for some $n\geq 0$ that $\lambda(h^n(\tt{0}))=\xi^n(\tt{0})$.  Using the fact that $\lambda\circ h=\xi\circ \lambda$, we obtain
\[
\lambda(h^{n+1}(\tt{0}))=\xi(\lambda(h^n(\tt{0})))=\xi(\xi^n(\tt{0}))=\xi^{n+1}(\tt{0}),
\]
which completes the proof of (1).

For (2), we show that $\Delta(f(g(h^n(\tt{0})))\tt{0})=\eta(\xi^{n+1}(\tt{0}))$ for every $n\geq 0$.  By Lemma~\ref{DeltaMorphisms}, we have $\Delta(f(g(h^n(\tt{0})))\tt{0})=\mu(h^n(\tt{0})),$ so it suffices to show that $\mu(h^n(\tt{0}))=\eta(\xi^{n+1}(\tt{0}))$.  Using the facts that $\mu=\eta\circ \xi\circ \lambda$ and $\lambda(h^n(\tt{0}))=\xi^n(\tt{0})$, we obtain
\[
\mu(h^n(\tt{0}))=\eta(\xi(\lambda(h^n(\tt{0}))))=\eta(\xi(\xi^{n}(\tt{0})))=\eta(\xi^{n+1}(\tt{0})),
\]
which completes the proof of (2).
\end{proof}

Since $\xi^\omega(\tt{0})$ and $\eta(\xi^\omega(\tt{0}))$ are Sturmian words, we have proved that $f(h^\omega(\tt{0}))$ and $f(g(h^\omega(\tt{0})))$ are complementary symmetric Rote words. Since all complementary symmetric Rote words are rich \cite[Theorem~25]{BMBLV11}, the following is immediate.

\begin{theorem}
The words $f(h^\omega(\tt{0}))$ and $f(g(h^\omega(\tt{0})))$ are rich.
\end{theorem}

Now we analyze the repetitions in $u=f(g(h^\omega(\tt{0})))$.
Let $v =\Delta(u)= \eta(\xi^\omega(\tt{0}))$ (by Lemma~\ref{SturmianMorphisms}).  The relation between the repetitions in $u$ and those in $v$ is given by the following lemma.

\begin{lemma}\label{reps_utov}
For any infinite binary word $x = (x_n)_{n \geq 0}$, let $y =(y_n)_{n \geq 0} = \Delta(x)$.  If $x$ contains a repetition
\[ 
(x_ix_{i+1} \cdots x_{i+\ell-1})^ex_ix_{i+1} \cdots x_{i+t-1} 
\]
for some positive integers $e \geq 2$, $\ell \geq 1$, and $t \leq \ell$, then $y$ contains a repetition
\[ 
(y_iy_{i+1} \cdots y_{i+\ell-1})^ey_iy_{i+1} \cdots y_{i+t-2} 
\]
where the number of $1$'s in $y_{i}y_{i+1} \cdots y_{i+\ell-1}$ is even.
\end{lemma}

\begin{proof}
The fact that $y$ contains such a repetition is immediate.  To see that the number of $1$'s in $y_iy_{i+1} \cdots y_{i+\ell-1}$ is even, note first that
\[
\sum_{j=0}^r y_{i+j} \bmod 2 = (x_i + x_{i+r+1}) \bmod 2.
\]
Hence if $x_{i} = x_{i+\ell}$, we have
\[
\sum_{j=0}^{\ell-1} y_{i+j}\bmod 2 =(x_i+x_{i+\ell})\bmod 2=0.
\]
It follows that the number of $1$'s in $y_{i}y_{i+1}\cdots y_{i+\ell-1}$ is even, as required.
\end{proof}

We now analyze the repetitions in $v$.  We first need to review some basic definitions from the theory of Sturmian words and the theory of continued fractions.  Consider a real number $\alpha$ with continued fraction expansion $\alpha = [d_0; d_1, d_2, d_3, \ldots]$, where $d_0=0$ and $d_i$ is a positive integer for all $i>0$.

The \emph{characteristic Sturmian word with slope $\alpha$} (see \cite[Chapter~9]{AS03}) is the infinite word $c_\alpha$ obtained as the limit of the sequence of \emph{standard words} $s_n$ defined by
\[
s_{0} = 0,\quad s_1 = 0^{d_1-1}1,\quad s_n = s_{n-1}^{d_n}s_{n-2},\quad n \geq 2.
\]
For $n \geq 2$, we also define the \emph{semi-standard words}
\[
s_{n,t} = s_{n-1}^ts_{n-2},
\]
for every $1 \leq t < d_n$.  The slope $\alpha$ is the frequency of $1$'s in $c_\alpha$.  It is known that any Sturmian word with the same frequency of $1$'s has the same set of factors as $c_\alpha$.

We also make use of the \emph{convergents} of $\alpha$, namely
\[
\frac{p_n}{q_n} = [0; d_1, d_2, d_3, \ldots, d_n],
\]
where 
\begin{align*}
p_{-2} = 0,\quad p_{-1} = 1,\quad p_n = d_np_{n-1} + p_{n-2} \text{
  for } n \geq 0;\\
q_{-2} = 1,\quad q_{-1} = 0,\quad q_n = d_nq_{n-1} + q_{n-2} \text{
  for } n \geq 0.
\end{align*}
Note that $|s_n| = q_n$ for $n \geq 0$.  We use the well-known fact that $q_{n-1}/q_n=[0;d_n,d_{n-1},\dots,d_1]$.

We can now prove the main theorem concerning the critical exponent of
$u$.

\begin{theorem}\label{rote_ce}
The critical exponent of $u$ is $2+\sqrt{2}/2$.
\end{theorem}

\begin{proof}
Let $\bar{\xi} : \Sigma_2^* \to \Sigma_2^*$ be the Sturmian morphism
defined by $\tt{0} \to \tt{01}$, and $\tt{1} \to \tt{001}$.  Let
$\bar{\eta} : \Sigma_2^* \to \Sigma_2^*$ be the Sturmian morphism
defined by $\tt{0} \to \tt{0}$ and $\tt{1} \to \tt{001}$.  Let
$\bar{v} = \bar{\eta}(\bar{\xi}^\omega(\tt{0}))$.  The morphisms
$\bar{\xi}$ and $\bar{\eta}$ are obtained by conjugating and
complementing $\xi$ and $\eta$, so the factors of $\bar{v}$ are
exactly the complements of the factors of $v$.  Clearly, the periods
and exponents of the repetitions in $v$ and $\bar{v}$ are identical,
so we analyze the repetitions in $\bar{v}$ instead.  To analyze the
repetitions in $\bar{v}$ it suffices to consider the repetitions in
the characteristic word with the same slope as $\bar{v}$.

The matrix of $\bar{\xi}$ is $M_{\bar{\xi}} =
\left(\begin{matrix}1&2\\1&1\end{matrix}\right)$ and the matrix of
  $\bar{\eta}$ is $M_{\bar{\eta}} =
  \left(\begin{matrix}1&2\\0&1\end{matrix}\right)$.  The frequency
    vector of $\tt{0}$'s and $\tt{1}$'s in $\xi^\omega(\tt{0})$ is the
    normalized eigenvector ${\bf v}$ of $M_{\bar{\xi}}$ corresponding
    to the dominant eigenvalue $1+\sqrt{2}$.  We have ${\bf v} =
    (2-\sqrt{2}, \sqrt{2}-1)^T$.  We then compute $M_{\bar{\eta}}{\bf
      v}$ and normalize to find that the frequency of $1$'s in
    $\bar{v}$ is $\alpha = (3-\sqrt{2})/7$.  We therefore consider the
    characteristic word $c_\alpha$ with slope $\alpha$ in place of
    $\bar{v}$.

Since $\alpha = [0; 4, \overline{2}]$, we see that
$c_\alpha$ is the infinite word obtained as the limit of the sequence
of standard words $s_k$ defined by
\[
s_0 = \tt{0},\quad s_1 = s_0^{4-1}\tt{1},\quad s_k = s_{k-1}^2s_{k-2},\quad k \geq 2
\]
We have $s_1=\tt{0001}$, $s_2 = \tt{000100010}$, $s_3 = \tt{0001000100001000100001}$,
etc.  We will also need the semi-standard words
\[
s_{k,1} = s_{k-1}s_{k-2},\quad k \geq 2.
\]
Note that the number of $\tt{0}$'s in $s_k$ is always odd and the number of
$\tt{0}$'s in $s_{k,1}$ is always even.  Also note that
by~\cite[Proposition~4.6.12]{Pel16}, the critical exponent of
$c_\alpha$ is $3+\sqrt{2}$.  Write $c_\alpha = (c_n)_{n \geq 0}$.

Now suppose that $u$ contains a repetition
\[ 
y^ey'=(u_iu_{i+1} \cdots u_{i+\ell-1})^eu_iu_{i+1} \cdots u_{i+t-1} 
\]
for some positive integers $e \geq 2$, $\ell \geq 1$, and $t \leq \ell$.  By Lemma~\ref{reps_utov}, we see that $v$ contains a repetition
\[ 
(v_iv_{i+1} \cdots v_{i+\ell-1})^ev_iv_{i+1} \cdots v_{i+t-2}, 
\]
where the number of $1$'s in $v_{i+1} \cdots v_{i+\ell-1}v_\ell$ is even.  It follows that $\bar{v}$, and hence $c_\alpha$, contains a repetition
\[ 
z^ez' = (c_jc_{j+1} \cdots c_{j+\ell-1})^ec_jc_{j+1} \cdots c_{j+t-2},
\]
where the number of $0$'s in $z$ is even.  The remainder of the argument is very similar to that of~\cite[Proposition~6]{RSV19}.

Suppose that $z$ is not primitive.  Since the critical exponent of
$c_\alpha$ is $3+\sqrt{2}$, the exponent of $z$ cannot be greater than
$2$.  Thus $z$ is a square, and we get that the exponent of $z^ez'$ is at
most
\[
\frac{3+\sqrt{2}}{2} < 2 + \frac{\sqrt{2}}{2}.
\]

So we may assume that $z$ is primitive.  By \cite[Corollary~4.6]{Pel15}
(originally due to Damanik and Lenz \cite{DL02}), the word $z$ is either a
conjugate of one of the standard words $s_k$, or a
conjugate of one of the semi-standard words $s_{k,1}$.  However, $s_k$
has an odd number of $\tt{0}$'s, so this case is ruled out.

Thus we may assume that $z$ is a conjugate of $s_{k,1}$ for some $k\geq 2$.  Hence $|z| = q_{k-2}+q_{k-1}$ for some
$k \geq 2$.  From~\cite[Theorem 4(i)]{Jus01}, one finds that the
longest factor of $c_\alpha$ with period $q_{k-2}+q_{k-1}$ has length
$2(q_{k-2}+q_{k-1})+q_{k-1}-2$.  It follows that $z^ez'$ has exponent
at most
\[
\frac{2(q_{k-2}+q_{k-1})+q_{k-1}-2}{q_{k-2}+q_{k-1}}
\]
for some $k\geq 2$.  In turn, it must be the case that $y^ey'$ has exponent
\begin{align}
E_k&= \frac{2(q_{k-2}+q_{k-1})+q_{k-1}-1}{q_{k-2}+q_{k-1}} \nonumber\\
&= 2 + \frac{q_{k-1}-1}{q_{k-2}+q_{k-1}}\label{semistd_exp_1}\\
&= 2 + \frac{1-1/q_{k-1}}{1+q_{k-2}/q_{k-1}} \label{semistd_exp_2}
\end{align}
for some $k\geq 2$.  

We claim that $\lim_{k\rightarrow\infty}E_k=2+\sqrt{2}/2$, and that the sequence $(E_k)_{k\geq 2}$ is increasing.  It follows that the exponent of $y^ey'$ is at most $2+\sqrt{2}/2$.  Moreover, by the discussion above, the word $u$ has a factor of exponent $E_k$ for every $k\geq 2$.  Thus, we conclude from the claim that $u$ has critical exponent $2+\sqrt{2}/2$.  We now complete the proof of the claim.

First we show that $\lim_{k\rightarrow\infty}E_k=2+\sqrt{2}/2$. Since $q_{k-2}/q_{k-1}=[0;\underbrace{2,2,\dots,2}_{k-2},4],$ we see immediately that $\lim_{k\rightarrow \infty} q_{k-2}/q_{k-1}=[0;\overline{2}]=\sqrt{2}-1.$  From~\eqref{semistd_exp_2}, we obtain
\[
\lim_{k\rightarrow\infty}E_k=2+\sqrt{2}/2.
\]

Finally, we show that the sequence $(E_k)_{k\geq 2}$ is increasing.  Let $k\geq 2$.  Starting from~\eqref{semistd_exp_1}, using algebra and the recursion $q_k=2q_{k-1}+q_{k-2}$, one finds that $E_{k+1}>E_k$ if and only if
\begin{align}\label{inc}
2q_{k-1}>q_{k-1}^2-q_kq_{k-2}.
\end{align}
When $k=2$, we have $q_{k-1}^2-q_{k}q_{k-2}=4^2-9\cdot 1=7.$  Suppose for some $k\geq 2$ that $q_{k-1}^2-q_kq_{k-2}=7(-1)^k$.  Then
\begin{align*}
q_{k}^2-q_{k+1}q_{k-1}&=(2q_{k-1}+q_{k-2})q_{k}-(2q_k+q_{k-1})q_{k-1}=q_{k-2}q_k-q_{k-1}^2=7(-1)^{k+1}.
\end{align*}
Thus, by mathematical induction, we have $q_{k-1}^2-q_kq_{k-2}=7(-1)^k$ for all $k\geq 2$.  In particular, the right-hand side of~\eqref{inc} is at most $7$ for all $k\geq 2$.  Since $q_{k-1}\geq q_1=4$ for all $k\geq 2$, we conclude that~\eqref{inc} is satisfied for all $k\geq 2$.  Therefore, we have $E_{k+1}>E_k$ for all $k\geq 2$.

This completes the proof of the claim, and hence the theorem.
\end{proof}

Since $f(h^\omega(\tt{0}))$ and $f(g(h^\omega(\tt{0})))$ both have critical exponent $2+\sqrt{2}/2$, Theorem~\ref{threshold} now follows immediately from Theorem~\ref{structure}.

% A similar analysis could be done for $f(h^\omega(\tt{0}))$, thereby
% giving another proof of the result of Baranwal and Shallit
% \cite{BS19}.

\section{Future Prospects}

For $k\geq 3$, it remains an open problem to determine the repetition threshold $\RRT(k)$ for the language of rich words on $k$ letters.  In fact, we even lack a conjecture for the value of $\RRT(k)$ in these cases.  Baranwal and Shallit~\cite{BS19} have established that $\RRT(3)\geq 9/4$, but did not explicitly conjecture that $\RRT(3)=9/4.$

\acknowledgements
After reading an early draft of this work, Edita Pelantov\'a pointed
out to us that $f(h^\omega(\tt{0}))$ and $f(g(h^\omega(\tt{0})))$ are
complementary symmetric Rote words.  We are very grateful to her for
this observation, as it allowed us to prove some stronger results than
were in our original draft.  We take this opportunity to refer the
reader to the recent manuscript by Medkov\'a, Pelantov\'a, and Vuillon
\cite{MPV18} for more results on complementary symmetric Rote words.
We thank Jeffrey Shallit for his comments on our earlier draft as
well.  We also thank the referee, who provided some helpful comments.

\bibliographystyle{abbrvnat}

% \bibliography{references.bib}

\end{document}